\documentclass[12pt, twoside, reqno, A4]{amsart}
\usepackage{amsmath}
\usepackage{hyperref}
\newtheorem{theorem}{Theorem}[section]
\newtheorem{corollary}[theorem]{Corollary}
\newtheorem{lemma}[theorem]{Lemma}
\theoremstyle{definition}

\theoremstyle{remark}

\begin{document}
\title{SOLUTION AND STABILITY OF A MIXED TYPE \\ FUNCTIONAL EQUATION}
\author{Pasupathi Narasimman$^1$ and Abasalt Bodaghi$^2$}
\maketitle
\begin{center}
\pagestyle{myheadings} \markboth{ P.  Narasimman, A. Bodaghi} {SOLUTION AND STABILITY OF A MIXED TYPE FUNCTIONAL EQUATION}
\thispagestyle{empty} \maketitle $^{1}$Department of Mathematics, Jeppiaar Institute of Technology,\\
Kunnam, Sriperumbudur(TK), Chennai- 631 604, India.\\
{\bf E-mail:~}\verb+drpnarasimman81@gmail.com+\\
\maketitle$^{2}$Department of Mathematics, Garmsar Branch,\\ Islamic Azad University, Garmsar, Iran.\\
{\bf E-mail:~}\verb+abasalt.bodaghi@gmail.com+
\end{center}
\begin{abstract}
In this paper, we obtain the general solution and investigate the generalized 
Hyers-Ulam-Rassias stability for the new mixed type additive and cubic functional equation 
\begin{align*}
 &3f(x+3y) - f(3x + y)\\&\quad=12[f(x+y)+f(x-y)]-16[f(x)+f(y)] + 12f(2y) - 4f(2x).
\end{align*}
As some corollaries, we show that the stability of this equation can be controlled by the sum and product of powers of norms.
\end{abstract}

\setcounter{equation}{0}
\section{Introduction}
In 1940, Ulam \cite{Ulam} raised the following fundamental question in the theory of functional equations concerning the group homomorphism: 
\par ``When is it true that a function, which approximately satisfies a functional equation must be close to an exact solution of the equation?"
\par One year later, Hyers \cite{Hyers} gave an affirmative solution to the above problem concerning the Banach space. The result of Hyers was generalized by Aoki \cite{Aoki} for approximate additive mappings and by Rassias \cite{THMR} for approximate linear mappings by allowing the difference Cauchy equation  $\left\| {f(x + y) - f(x) - f(y)} \right\|$ to be controlled by 
$\varepsilon (\left\| x \right\|^p  + \left\| y \right\|^p )$. The stability phenomenon that was proved by Rassias \cite{THMR} is called the Hyers-Ulam-Rassias stability. 
\par In 1982-1989, J. M. Rassias \cite{ra2, ra3, ra4, ra5} generalized the Hyers stability result
by presenting a weaker condition controlled by a product of different powers of norms. In fact, he proved the following theorem:
\begin{theorem} Let $X$ be a real normed linear space and $Y$ a real
complete normed linear space. Assume that $f:X \longrightarrow Y$  is an
approximately additive mapping for which there exist constants
$\theta \geq0$ and $p,q \in \mathbb{R}$ such that $r=p+q \neq 1$ and
$f$ satisfies inequality
\begin{equation}\label{1.6}
\left\| {f(x + y) - f(x) - f(y)} \right\| \le \theta \left\| x
\right\|^{^p } \left\| y \right\|^{^q }
\end{equation}
for all $x,y \in X$. Then there exist a unique additive mapping $L:E
\longrightarrow E'$ satisfying
\begin{equation}\label{1.7}
\left\| {f(x) - L(x)} \right\| \le \frac{\theta }{{\left| {2^r  - 2}
\right|}}\left\| x \right\|^r
\end{equation}
for all $x \in X$. If, in addition, $f:X\longrightarrow Y$  is a mapping such
that the transformation $t \longmapsto f(tx)$  is continuous in $t\in
\mathbb{R}$  for each fixed  $x \in X$, then $L$ is an
$\mathbb{R}$-linear mapping.
\end{theorem}
\par In 1994, a generalization of Rassias \cite{THMR} theorem was obtained by G\v avruta \cite{Gavruta}, who replaced $\varepsilon (\left\| x \right\|^p  + \left\| y \right\|^p )$ by a general control function $\phi(x,y)$. This idea is known as generalized Hyers-Ulam-Rassias stability. Since then, general stability problems of various functional equations such as quadratric, cubic, quartic, quintic, sexitic and mixed type of such functional equations and also Pexiderized versions with more general domains and ranges have been investigated by a number of authors \cite{bod1, boda2, Czerwik, Jun01, Jun05, Najati}.

Hyers-Ulam-Rassias stability for a mixed quadratic and additive functional equation
\begin{equation}\label{1.3}
 f(x+2y)+f(x-2y)+4f(x)=3\left[f(x+y)+f(x-y)\right]+f(2y)-2f(y)
\end{equation}
in quasi Banach Space was dealt by Moradlou et al., in \cite{FM}.
It is easy to see that the function $f(x)=ax^2+bx$ is a solution of the functional equation (\ref{1.3}); for the general case of  (\ref{1.3}) see \cite{bk1} and \cite{bk2}.

 In 2001, J. M. Rassias \cite{ra1}, introduced the cubic functional equation
\begin{eqnarray}\label{r}f(x+2y)-3f(x+y)+3f(x)-f(x-y)=6f(y)\end{eqnarray}
and established the solution of the Ulam-Hyers stability problem for these cubic mappings. It may be noted that the function $f(x)= cx^3$ is a solution of the functional equation (\ref{r}). Hence, every solution of the cubic functional equation (\ref{r}) is said to be a  cubic function. Other versions of a  cubic functional equation can be found in  \cite{boda}, \cite{bmr}, \cite{Junkim02} and \cite{jk}. 
\par In 2010, J. M. Rassias et al., \cite{rras}, found the general solution and Ulam stability of mixed type cubic and additive functional equation of the form
\begin{eqnarray}\label{rr}
&3f(x+y+z)+f(-x+y+z)+f(x-y+z)+f(x+y-z)\\&\qquad\notag+4[f(x)+f(y)+f(z)]=4[f(x+y)+f(x+z)+f(y+z)]
\end{eqnarray}
They also studied the stability of the equation (\ref{rr}) controlled by a mixed type product-sum of powers of norms. In the same year, K. Ravi et al., \cite{rs}, investigated the general solution and Ulam stability of mixed type cubic and additive functional equation (\ref{rr}) in fuzzy normed spaces.
 
\par In this paper, we discuss a new additive and cubic type functional equation of the form
\begin{align}\label{1.4}
&3f(x+3y) - f(3x + y)\notag \\&\quad=12[f(x+y)+f(x-y)]-16[f(x)+f(y)] + 12f(2y) - 4f(2x)
\end{align}
and obtain its general solution and also investigate its generalized Hyers-Ulam-Rassias stability in Banach spaces. Finally, we prove that the stability of the equation (\ref{1.4}) can be controlled by the sum and product of powers of norms.


\setcounter{equation}{0}
\section{General Solution of the Functional Equation (\ref{1.4})}
Throughout this section, we assume that $X$ and $Y$ are linear spaces. We will find
out the general solution of (\ref{1.4}). Firstly, we indicate two following lemmas which play fundamental role to reach our goal.
\begin{lemma}\label{lem1}
A mapping $ f:X \longrightarrow Y $ is additive if and only
if $f$ satisfies the functional equation
\begin{equation}\label{2.1}
3f(x+3y) - f(3x + y) = 12[f(x+y) + f(x-y)]-24f(x)+8f(y)
\end{equation}
for all $x,y\in X$.
\end{lemma}
\begin{proof} Suppose that $f$ is additive. Then,  the
standard additive functional equation
\begin{equation}\label{2.2}
f(x+y)=f(x)+f(y)
\end{equation}
holds for all $x,y\in X$. Putting $x=y=0$ in (\ref{2.2}), we see that
$f(0)=0$, and setting $(x,y)$ by $(x,x)$ in (\ref{2.2}), we obtain
\begin{equation}\label{2.3}
f(2x) = 2f(x)
\end{equation}
for all $x\in X$. Replacing $y$ by $2x$ in (\ref{2.2}) and using (\ref{2.3}),
we get
\begin{equation}\label{2.4}
f(3x) = 3f(x)
\end{equation}
for all $x\in X$. Interchanging $y$ into $-x$ in (\ref{2.2}), we arrive at
\[
f(-x)=-f(x)
\]
for all $x\in X$. Consequently, $f$  is odd. Setting $(x,y)$ by $(x+y,x-y)$ in
(2.2) and multiply the resultant by 12, we have
\begin{equation}\label{2.5}
24f(x)=12[f(x+y)+f(x-y)]
\end{equation}
for all $x,y\in X$. Switching $(x,y)$ to $(x,3y)$ in
(\ref{2.2}) and multiplying  the resultant by 3, we get
\begin{equation}\label{2.6}
3f(x+3y)=3f(x)+9f(y)
\end{equation}
for all $x,y\in X$. Substituting $x$ by $3x$ in
(\ref{2.2}), we obtain
\begin{equation}\label{2.7}
f(3x+y)=3f(x)+f(y)
\end{equation}
for all $x,y\in X$. Subtracting (\ref{2.7}) from (\ref{2.6}), we deduce that
\begin{equation}\label{2.8}
3f(x+3y)-f(3x+y)=8f(y)
\end{equation}
for all $x\in X$. Adding (\ref{2.5}) and (\ref{2.8}), we arrive (\ref{2.1}). 

Conversely, assume that $f$ satisfies the functional equation (\ref{2.1}). Setting $(x,y)=(0,0)$ and $(x,0)$ in (\ref{2.1}), we get $f(0) = 0$ and 
\begin{equation}\label{2.9}
f(3x) = 3f(x)
\end{equation}
respectively, for all $x \in X$. Replacing $(x,y)$ by $(0,x)$ in (\ref{2.1}) and using
(\ref{2.9}), we obtain
\begin{equation}\label{2.10}
f(-x)=-f(x)
\end{equation}
for all $x\in X$. Thus $f$ is an odd function. Letting $(x,y)=(-x,x)$  in (\ref{2.1}) and
applying (\ref{2.10}), we have
\begin{equation}\label{2.11}
f(2x)=2f(x)
\end{equation}
for all $x\in X$. Interchanging $(x,y)$ into $(x-y,x+y)$ in (\ref{2.1}), we get
\begin{equation}\label{2.12}
3f(4x+2y)-f(4x-2y)=24[f(x)-f(y)]-24f(x-y)+8f(x+y)
\end{equation}
for all $x,y\in X$. Replacing $y$  by $-y$ in (\ref{2.12}), using the oddness of $f$, and then adding the resultant
equation to (\ref{2.12}) and again applying (\ref{2.11}), we arrive at
\begin{equation}\label{2.13}
f(2x+y)+f(2x-y)=12f(x)-4f(x+y)-4f(x-y)
\end{equation}
for all $x,y\in X$. Setting $(x,y)=(x,-2y)$ in (\ref{2.13}) and using (\ref{2.11}), we get
\begin{equation}\label{2.14}
f(x - y) + f(x + y) = 6f(x) - 2f(x-2y)-2f(x+2y)
\end{equation}
for all $x,y\in X$. Replacing $x,y$, by $y,x$ in (\ref{2.14}), respectively, and using the oddness of $f$, we obtain
\begin{equation}\label{2.15}
-f(x - y) + f(x + y) = 6f(y) + 2f(2x-y)-2f(2x+y)
\end{equation}
for all $x,y\in X$. Switching $y$ to $-y$ in (\ref{2.15}),  we deduce that
\begin{equation}\label{2.16}
2f(2x - y)= 2f(2x + y)-6f(y) - f(x-y)+f(x+y)
\end{equation}
for all $x,y\in X$. Plugging (\ref{2.16}) into (\ref{2.13}), we have
\begin{equation}\label{2.17}
4f(2x + y)= -9f(x + y)-7f(x-y) + 24f(x)+6f(y)
\end{equation}
for all $x,y\in X$. Replacing $x+y$ by $y$ in (\ref{2.17}), we get
\begin{equation}\label{2.18}
7f(2x - y)= -4f(x + y)-6f(x-y) - 9f(y)+24f(x)
\end{equation}
for all $x,y\in X$. Adding (\ref{2.17}) and (\ref{2.18}), we obtain
\begin{equation}\label{2.19}
f(2x + y)+f(2x-y)= -\frac{79}{28}f(x + y)-\frac{73}{28}f(x-y)+\frac{6}{28}f(y)+\frac{264}{28}f(x)
\end{equation}
for all $x,y\in X$. Now, it follows from (\ref{2.13}) and (\ref{2.19}) that
\begin{equation}\label{2.20}
-11f(x + y)-13f(x-y)= 2f(y)-24f(x)
\end{equation}
for all $x,y\in X$. Replacing $x$ by $2x$ in (\ref{2.13}) and in the resultant again using (\ref{2.13}), we get
\begin{equation}\label{2.21}
f(4x + y)+f(4x-y)= -24f(x)+16f(x+y)+16f(x-y)
\end{equation}
for all $x,y\in X$. Putting $2x+y$ instead of $y$ in (\ref{2.13}) and applying the oddness of $f$, we obtain
\begin{equation}\label{2.22}
f(4x + y)-f(y)= 12f(x)-4f(3x+y)+4f(x+y)
\end{equation}
for all $x,y\in X$. Replacing $y$  by $-y$ in (\ref{2.22}), using oddness of $f$ and adding the resultant
equation with (\ref{2.22}), we have
\begin{equation}\label{2.23}
f(4x + y) + f(4x - y) = 24f(x) - 4[f(3x+y)+f(3x-y)]+4[f(x+y)+f(x-y)]
\end{equation}
for all $x,y\in X$. Substituting $y$ by $x+y$ in (\ref{2.13}) and using oddness of $f$, we arrive
\begin{equation}\label{2.24}
f(3x + y) + f(x - y) = 12f(x) - 4f(2x+y)+4f(y)
\end{equation}
for all $x,y\in X$. Replacing $y$ by $-y$ in (\ref{2.24}) and combining the resultant equation with (\ref{2.24}) and using (\ref{2.13}), we get
\begin{equation}\label{2.25}
f(3x + y)+f(3x-y)= -24f(x)+15f(x+y)+15f(x-y)
\end{equation}
for all $x,y\in X$. Plugging (\ref{2.25}) into (\ref{2.23}), we have
\begin{equation}\label{2.26}
f(4x+y)+f(4x-y)=120f(x)-56f(x+y)-56f(x-y)
\end{equation}
for all $x,y\in X$.  Combining the equations (\ref{2.21}) and (\ref{2.26}) to obtain 
\begin{equation}\label{2.27}
f(x-y)=2f(x)-f(x+y)
\end{equation}
for all $x,y\in X$. The equations (\ref{2.27}) and (\ref{2.20})  necessities (\ref{2.2}).
Therefore, $f$  is additive function.
\end{proof}
The following lemma is proved in \cite[Theorem 2.2]{RaviPN}.
\begin{lemma}\label{lem2}
A mapping $ f:X \longrightarrow Y $ is cubic if and only
if $f$ satisfies the functional equation
\begin{equation}\label{2.28}
3f(x+3y)-f(3x+y)=12[f(x+y)+f(x-y)]-48f(x)+80f(y)
\end{equation}
for all $x,y\in X$.
\end{lemma}

\begin{theorem}
A function $f:X \longrightarrow Y$ satifies the equation \emph{(\ref{1.4})} for all $x,y \in X$ if and only if then there exist  additive functions  $A,H:X \longrightarrow Y$ and cubic functions $C,G:X \longrightarrow Y$ such that $f(x)=H(x)+G(x)$ for all $x \in X$, where $H(x)=A(2x)-8A(x)$ and $G(x)=C(2x)-2C(x)$.
\end{theorem}

\begin{proof}
Define the mappings $A,H,C,G : X \longrightarrow Y$ via 
$$A(x),C(x) := \frac{f(x)-f(-x)}{2}$$ and $$H(x):=A(2x)-8A(x), G(x):=C(2x)-2C(x)$$
for all $x \in X$. Then, we have 
$$A(0)=0,\,\, C(0)=0,\,\, A(-x)=-A(x),\,\, C(-x)=-C(x),$$
\begin{align}\label{2.29}
&3\mathfrak A(x+3y)-\mathfrak A(3x+y)\notag \\&\quad =12[\mathfrak A(x+y)+\mathfrak A(x-y)]-16[\mathfrak A(x)+\mathfrak A(y)]+12\mathfrak A(2y)-4\mathfrak A(2x)
\end{align}
for all $x,y \in X$ in which $\mathfrak A\in \{A,C,H,G\}$. First, we claim that $H$ is additive. Setting $(x,y)$ by $(x,x)$ in (\ref{2.29}), we have
\begin{align}\label{2.33}
A(4x)= 10A(2x)-16A(x)
\end{align}
for all $x \in X$. Using $H(x)=A(2x)-8A(x)$ in (\ref{2.33}), we get
$$H(2x)= 2H(x).$$
Therefore the equation (\ref{2.29}) when $\mathfrak A=H$  is reduced to the form 
\[
3H(x+3y)-H(3x+y)=12[H(x+y)+H(x-y)]-24H(x)+8H(y)
\]
for all $x,y \in X$ and hence Lemma \ref{lem1} guarantees that $H$ is additive. Secondly, we claim that $G$ is cubic. Letting $(x,y)$ by $(x,x)$ in (\ref{2.29}) for $\mathfrak A=C$, we obtain
\begin{align}\label{2.34}
C(4x)= 10C(2x)-16C(x)
\end{align}
for all $x \in X$. Using $G(x)=C(2x)-2C(x)$ in (\ref{2.34}), we have
$$G(2x)= 8G(x).$$
So the equation (\ref{2.29}) can be rewritten as  
\[
3G(x+3y) - G(3x + y) = 12[G(x+y) + G(x-y)]-48G(x)+80G(y)
\]
for all $x,y \in X$ when $\mathfrak A=G$. Now, Lemma \ref{lem2} implies that $G$ is cubic. Hence, $f:X \longrightarrow Y$ satifies the equation (\ref{1.4}), and thus $f(x)=H(x)+G(x)$ for all $x \in X$.
\par Conversely, suppose that there exist additve functions  $A,H:X \longrightarrow Y$ and a cubic functions  $C,G:X \longrightarrow Y$ such that $f(x)=H(x)+G(x)$ for all $x \in X$. We have
$$H(2x)= 2H(x)~~~~~~ \text{and} ~~~~~~~G(2x)= 8G(x)$$
for all $x \in X$.  It follows from Lemmas \ref{lem1} and \ref{lem2} that 
\begin{align*}
3&f(x+3y)-f(3x+y)-12[f(x+y)+f(x-y)]+16[f(x)+f(y)] \\
&- 12f(2y) + 4f(2x)=3H(x+3y) - H(3x + y)-12[H(x+y)+H(x-y)]\\
&+16[H(x)+H(y)] - 12H(2y) + 4H(2x)
+3G(x+3y) - G(3x + y)\\
&-12[G(x+y)+G(x-y)]+16[G(x)+G(y)] - 12G(2y) + 4G(2x)
\\&=3H(x+3y) - H(3x + y)-12[H(x+y)+H(x-y)]+24H(x)-8H(y)\\
&+3G(x+3y) - G(3x + y)-12[G(x+y)+G(x-y)]+48G(x)-80G(y)=0
\end{align*}
for all $x,y \in X$. This finishes the proof.
\end{proof}
\setcounter{equation}{0}
\section{Stability of the Functional Equation (\ref{1.4})}
From now on, we assume that $X$ is a normed space and $Y$ is a Banach space. For convenience, we use the following difference operator for a given mapping $ f:X \longrightarrow Y$ 
\[
\begin{array}{l}
 D_f(x,y)=3f(x+3y)-f(3x+ y) -12[f(x+y)+f(x-y)] \\
 \qquad \,\,\,\,\,\,\,\,\,\,\,\,+16[f(x)+f(y)]-12f(2y)+4f(2x)
 \end{array}
\]
for all $x,y\in X$. In the upcoming result, we investigate the generalized Hyers-Ulam-Rassias stability problem for functional equation (\ref{1.4}).


\begin{theorem}\label{th2}
Let $l\in\{-1,1\}$. Suppose that an odd mapping $f: X \longrightarrow Y$ satisfies
\begin{equation}\label{3.1}
\| D_f(x,y) \|  \leq \phi(x,y)
\end{equation}
for all $x,y \in X$. If $\phi:X \times X \longrightarrow [0,\infty )$ is a function such that
\begin{equation}\label{3.2}
\sum_{i=1}^{\infty} 2^{il}\phi\left(\frac{x}{2^{il}},\frac{x}{2^{il}}\right)<\infty
\end{equation}
for all $x \in X$ and that $\lim_n 2^{ln}\phi(\frac{x}{2^{ln}},\frac{y}{2^{ln}})=0$ for all $x,y \in X$, then the limit
\[
A(x) = \lim_{n\rightarrow\infty} 2^{ln}\left[f\left(\frac{x}{2^{l(n-l)}}\right)-8f\left(\frac{x}{2^{ln}}\right)\right] 
\]
exists for all $x \in X$, and $A: X \longrightarrow Y$ is a unique additive function satisfies \emph{(\ref{1.4})} and
\begin{equation}\label{3.3}
\left\| f(2x)-8f(x)-A(x) \right\| \leq \frac{1}{2}\sum\limits_{i =\frac{|l-1|}{2}}^{\infty} 2^{il}\phi\left(\frac{x}{2^{l(i+l)}},\frac{x}{2^{l(i+l)}}\right)
\end{equation}
for all $x \in X$.
\end{theorem}

\begin{proof}
Replacing $y$ by $x$ in (\ref{3.1}), we have
\begin{equation}\label{3.4}
\left\| f(4x)-10f(2x)+16f(x) \right\| \leq \frac{1}{2}\phi(x,x)
\end{equation}
for all $x \in X$.  By (\ref{3.4}) we have
\begin{equation}\label{3.5}
\left\| H(2x)-2H(x)\right\| \leq \frac{1}{2}\phi(x,x)
\end{equation}
for all $x \in X$, where $H(x)=f(2x)-8f(x)$. It follows from (\ref{3.5})
 \begin{equation}\label{3.6}
\left\| H(x)-2^lH\left(\frac{x}{2^l}\right)\right\| \leq \frac{1}{2^{\frac{|l-1|+2}{2}}}\phi\left(\frac{x}{2^{\frac{l+1}{2}}},\frac{x}{2^{\frac{l+1}{2}}}\right)
\end{equation}
for all $x \in X$. Again, by switching $x$ to $\frac{x}{2^l}$ in (\ref{3.6}) and combining the resultant equation with (\ref{3.6}), we get
 \begin{equation*}
\left\| H(x)-2^{2l}H\left(\frac{x}{2^{2l}}\right)\right\| \leq \frac{1}{2}\,\sum\limits_{i = \frac{|l-1|}{2}}^{2-\frac{|l+1|}{2}} 2^{il}\phi\left(\frac{x}{2^{l(i+l)}},\frac{x}{2^{l(i+l)}}\right)
\end{equation*}
for all $x \in X$. An induction argument now implies that 
\begin{equation}\label{3.7}
\left\| H(x)-2^{ln}H\left(\frac{x}{2^{ln}}\right)\right\| \leq \frac{1}{2}\,\sum\limits_{i = \frac{|l-1|}{2}}^{n-\frac{|l+1|}{2}} 2^{il}\phi\left(\frac{x}{2^{l(i+l)}},\frac{x}{2^{l(i+l)}}\right)
\end{equation}
for all $x \in X$. Multiplying both sides of inequality (\ref{3.7}) by $2^{lm}$ and $x$ by $\frac{x}{2^{lm}}$, we get
\begin{align*}
\left\| 2^{lm}H\left(\frac{x}{2^{ml}}\right)-2^{(m+n)l}H\left(\frac{x}{2^{(m+n)l}}\right) \right\|& \leq \frac{1}{2}\,\sum_{i=\frac{|l-1|}{2}}^{n-\frac{|l+1|}{2}} 2^{(i+m)l}\phi\left(\frac{x}{2^{(i+m+l)l}},\frac{x}{2^{(i+m+l)l}}\right)\\ 
&=\frac{1}{2}\,\sum_{i=m+\frac{|l-1|}{2}}^{m+n-\frac{|l+1|}{2}} 2^{il}\phi\left(\frac{x}{2^{(i+l)l}},\frac{x}{2^{(i+l)l}}\right).
\end{align*}
Since the right hand side of the above inequality tends to $0$ as $m \to \infty$, the sequence $\{2^{ln}H(\frac{x}{2^{ln}})\}$
is Cauchy. Then the limit 
$$A(x)=\lim_{n\rightarrow\infty} 2^{ln}H\left(\frac{x}{2^{ln}}\right)=\lim_{n\rightarrow\infty} 2^{ln}\left(f\left(\frac{x}{2^{l(n-l)}}\right)-8f\left(\frac{x}{2^{ln}}\right)\right)$$
exist for all $x \in X$. On the other hand, we have
\begin{align}\label{3.8}
\nonumber&\left\| A(2x)-2A(x)\right\|\\
\nonumber&=\lim_{n\rightarrow\infty}\left[2^{ln}\left(f\left(\frac{2x}{2^{(n-l)l}}\right)-8f\left(\frac{2x}{2^{ln}}\right)\right)-2\left(2^{ln}\left(f\left(\frac{x}{2^{(n-l)l}}\right)-8f\left(\frac{x}{2^{ln}}\right)\right)\right)\right]\\ 
\nonumber&=\lim_{n\rightarrow\infty}\left[2^{ln}H\left(\frac{x}{2^{l(n-l)}}\right)-2^{l(n+l)}H\left(\frac{x}{2^{ln}}\right)\right]\\
&=2\lim_{n\rightarrow\infty}\left[2^{l(n-l)}H\left(\frac{x}{2^{l(n-l)}}\right)-2^{ln}H\left(\frac{x}{2^{ln}}\right)\right]=0
\end{align}
for all $x \in X$. Let $D_H(x,y)=D_f(2^lx,2^ly)-8D_f(x,y)$ for all $x \in X$. Then we have
\begin{align*}
D_A(x,y)&=\lim_{n\rightarrow\infty}\left\|2^{ln} D_H\left(\frac{x}{2^{ln}},\frac{y}{2^{ln}}\right)\right\|\\
&=\lim_{n\rightarrow\infty}2^{ln}\left\|D_f\left(\frac{x}{2^{l(n-l)}},\frac{y}{2^{l(n-l)}}\right)-8D_f\left(\frac{x}{2^{ln}},\frac{y}{2^{ln}}\right)\right\|\\
&\leq \lim_{n\rightarrow\infty} 2\left\|2^{l(n-l)} D_f\left(\frac{x}{2^{l(n-l)}},\frac{y}{2^{l(n-l)}}\right)\right\|+\lim_{n\rightarrow\infty} 8\left\|2^{ln} D_f\left(\frac{x}{2^{ln}},\frac{y}{2^{ln}}\right)\right\|\\
&\leq 2 \lim_{n\rightarrow\infty} 2^{l(n-l)} \phi\left(\frac{x}{2^{l(n-l)}},\frac{y}{2^{l(n-l)}}\right)+8 \lim_{n\rightarrow\infty} 2^{ln} \phi\left(\frac{x}{2^{ln}},\frac{y}{2^{ln}}\right)=0.
\end{align*}
This means that $A$ satisfies (\ref{1.3}). By (\ref{3.8}), it follows that $A$ is additive. It remains to show that $A$ is unique additive function which satisfies (\ref{3.3}). Suppose that there exists another additive function $A':X \longrightarrow Y$ satisfies (\ref{3.3}). We have $A(2^{ln}x)=2^{ln}A(x)$ and $A'(2^{ln}x)=2^{ln}A'(x)$ for all $x \in X$. The last equalities imply that
\begin{align*}
&\left\| A(x)-A'(x) \right\|\\
&\quad=2^{ln}\left\|A\left(\frac{x}{2^{ln}}\right)-A'\left(\frac{x}{2^{ln}}\right)\right\|\\
&\quad\leq 2^{ln}\left[\left\|A\left(\frac{x}{2^{ln}}\right)-f\left(\frac{2x}{2^{ln}}\right)-8f\left(\frac{x}{2^{ln}}\right)\right\|
+\left\|A'\left(\frac{x}{2^{ln}}\right)-f\left(\frac{2x}{2^{ln}}\right)-8f\left(\frac{x}{2^{ln}}\right)\right\|\right]\\&\quad
\leq \sum_{i=n+\frac{|l-1|}{2}}^{\infty} 2^{il} \phi\left(\frac{x}{2^{l(i+l)}},\frac{x}{2^{l(i+l)}}\right)
\end{align*}
for all $x \in X$. Taking $n \to \infty$, we see that the right hand side of above inequality tends to $0$. Thus, $A(x)=A'(x)$ for all $x \in X$. This completes the proof.
\end{proof}

We have the following result which is analogous to Theorem \ref{th2} for another case of $f$. The method is similar but we bring it.

\begin{theorem}\label{th3}
Let $l\in\{-1,1\}$. Suppose that an odd mapping $f: X \longrightarrow Y$ satisfying
\begin{equation}\label{3.9}
\| D_f(x,y) \|  \leq \phi(x,y)
\end{equation}
for all $x,y \in X$. If $\phi:X \times X \longrightarrow [0,\infty )$ is a function such that
\begin{equation}\label{3.10}
\sum_{i=1}^{\infty} 8^{il}\phi\left(\frac{x}{2^{il}},\frac{x}{2^{il}}\right)<\infty
\end{equation}
for all $x \in X$ and that $\lim_n 8^{ln}\phi(\frac{x}{2^{ln}},\frac{y}{2^{ln}})=0$ for all $x,y \in X$, then the limit
\[
C(x) = \lim_n 8^{ln}\left[f\left(\frac{x}{2^{l(n-l)}}\right)-2f\left(\frac{x}{2^{ln}}\right)\right] 
\]
exists for all $x \in X$, and $C: X \longrightarrow Y$ is a unique cubic mapping satisfies \emph{(\ref{1.4})} and also
\begin{equation}\label{3.11}
\left\| f(2x)-2f(x)-C(x) \right\| \leq \frac{1}{2}\sum\limits_{i =\frac{|l-1|}{2}}^{\infty} 8^{il}\phi\left(\frac{x}{2^{l(i+l)}},\frac{x}{2^{l(i+l)}}\right)
\end{equation}
for all $x \in X$.
\end{theorem}
\begin{proof}
Interchanging $y$ into $x$ in (\ref{3.9}), we have
\begin{equation}\label{3.12}
\left\| f(4x)-10f(2x)+16f(x) \right\| \leq \frac{1}{2}\phi(x,x)
\end{equation}
for all $x \in X$. Put $G(x)=f(2x)-2f(x)$ for all $x \in X$. By (\ref{3.12}), we obtain
\begin{equation}\label{3.13}
\left\| G(2x)-8G(x)\right\| \leq \frac{1}{2}\phi(x,x)
\end{equation}
for all $x \in X$. Replacing $x$ by $\frac{x}{2}$ in (\ref{3.13}), we get
 \begin{equation}\label{3.14}
\left\| G(x)-8^lG\left(\frac{x}{2^l}\right)\right\| \leq \frac{1}{2}\frac{1}{8^{\frac{|l-1|}{2}}}\phi\left(\frac{x}{2^{\frac{l+1}{2}}},\frac{x}{2^{\frac{l+1}{2}}}\right)
\end{equation}
for all $x \in X$. Once more, by replacing $x$ by $\frac{x}{2^l}$ in (\ref{3.14}) and combining the resultant equation with (\ref{3.14}), we deduce that
 \begin{equation*}
\left\| G(x)-8^{2l}G\left(\frac{x}{2^{2l}}\right)\right\| \leq \frac{1}{2}\,\sum\limits_{i = \frac{|l-1|}{2}}^{2-\frac{|l+1|}{2}} 8^{il}\phi\left(\frac{x}{2^{l(i+l)}},\frac{x}{2^{l(i+l)}}\right)
\end{equation*}
for all $x \in X$. The above process can be repeated to obtain 
\begin{equation}\label{3.15}
\left\| G(x)-8^{ln}G\left(\frac{x}{2^{ln}}\right)\right\| \leq \frac{1}{2}\,\sum\limits_{i = \frac{|l-1|}{2}}^{n-\frac{|l+1|}{2}} 8^{il}\phi\left(\frac{x}{2^{l(i+l)}},\frac{x}{2^{l(i+l)}}\right)
\end{equation}
for all $x \in X$. Multiplying both sides of inequality (\ref{3.15}) by $8^{lm}$ and $x$ by $\frac{x}{2^{lm}}$, we have
\begin{align*}
\left\| 8^{lm}G\left(\frac{x}{2^{ml}}\right)-8^{(m+n)l}G\left(\frac{x}{2^{(m+n)l}}\right) \right\|& \leq \frac{1}{2}\,\sum_{i=\frac{|l-1|}{2}}^{n-\frac{|l+1|}{2}} 8^{(i+m)l}\phi\left(\frac{x}{2^{(i+m+l)l}},\frac{x}{2^{(i+m+l)l}}\right)\\ 
&=\frac{1}{2}\,\sum_{i=m+\frac{|l-1|}{2}}^{m+n-\frac{|l+1|}{2}} 8^{il}\phi\left(\frac{x}{2^{(i+l)l}},\frac{x}{2^{(i+l)l}}\right).
\end{align*}
Letting $m \to \infty$ in the above relation, we see that $\{8^{n}G(\frac{x}{2^n})\}$ is a Cauchy sequence. Due to the completeness of $Y$, this sequence is convergent to $C(x)$. In other words,

$$C(x)=\lim_{n\rightarrow\infty} 8^{ln}G\left(\frac{x}{2^{ln}}\right)=\lim_{n\rightarrow\infty} 8^{ln}\left(f\left(\frac{x}{2^{l(n-l)}}\right)-2f\left(\frac{x}{2^{ln}}\right)\right)\quad(x \in X).$$

We now have
\begin{align}\label{3.16}
\nonumber&\left\| C(2x)-8C(x)\right\|\\
\nonumber&=\lim_{n\rightarrow\infty}\left[8^{ln}\left(f\left(\frac{2x}{2^{(n-l)l}}\right)-2f\left(\frac{2x}{8^{ln}}\right)\right)-8\left(8^{ln}\left(f\left(\frac{x}{2^{(n-l)l}}\right)-2f\left(\frac{x}{2^{ln}}\right)\right)\right)\right]\\ 
\nonumber&=\lim_{n\rightarrow\infty}\left[8^{ln}G\left(\frac{x}{2^{l(n-l)}}\right)-8^{l(n+l)}G\left(\frac{x}{2^{ln}}\right)\right]\\
&=8\lim_{n\rightarrow\infty}\left[8^{l(n-l)}G\left(\frac{x}{2^{l(n-l)}}\right)-8^{ln}G\left(\frac{x}{2^{ln}}\right)\right]=0
\end{align}
for all $x \in X$. Putting $D_G(x,y)=D_f(2^lx,2^ly)-2D_f(x,y)$ for all $x \in X$, we obtain
\begin{align*}
D_C(x,y)&=\lim_{n\rightarrow\infty}\left\|8^{ln} D_G\left(\frac{x}{2^{ln}},\frac{y}{2^{ln}}\right)\right\|\\
&=\lim_{n\rightarrow\infty}8^{ln}\left\|D_f\left(\frac{x}{2^{l(n-l)}},\frac{y}{2^{l(n-l)}}\right)-2D_f\left(\frac{x}{2^{ln}},\frac{y}{2^{ln}}\right)\right\|\\
&\leq \lim_{n\rightarrow\infty} 8\left\|8^{l(n-l)} D_f\left(\frac{x}{2^{l(n-l)}},\frac{y}{2^{l(n-l)}}\right)\right\|+\lim_{n\rightarrow\infty} 2\left\|8^{ln} D_f\left(\frac{x}{2^{ln}},\frac{y}{2^{ln}}\right)\right\|\\
&\leq 8 \lim_{n\rightarrow\infty} 8^{l(n-l)} \phi\left(\frac{x}{2^{l(n-l)}},\frac{y}{2^{l(n-l)}}\right)+2 \lim_{n\rightarrow\infty} 8^{ln} \phi\left(\frac{x}{2^{ln}},\frac{y}{2^{ln}}\right)=0.
\end{align*}
Hence, $C$ satisfies (\ref{1.4}). Now the relation (\ref{3.16}) showes that $C$ is cubic. If there exists a cubic mapping $C':X \longrightarrow Y$ satisfies (\ref{3.11}), then we have
\begin{align*}
&\left\| C(x)-C'(x) \right\|\\
&\quad=8^{ln}\left\|C\left(\frac{x}{2^{ln}}\right)-C'\left(\frac{x}{2^{ln}}\right)\right\|\\
&\quad\leq 8^{ln}\left[\left\|C\left(\frac{x}{2^{ln}}\right)-f\left(\frac{2x}{2^{ln}}\right)-2f\left(\frac{x}{2^{ln}}\right)\right\|
+\left\|C'\left(\frac{x}{2^{ln}}\right)-f\left(\frac{2x}{2^{ln}}\right)-2f\left(\frac{x}{2^{ln}}\right)\right\|\right]\\&\quad
\leq \sum_{i=n+\frac{|l-1|}{2}}^{\infty} 8^{il} \phi\left(\frac{x}{2^{l(i+l)}},\frac{x}{2^{l(i+l)}}\right)
\end{align*}
for all $x \in X$. The right hand side of above inequality goes to $0$ as $n \to \infty$. Thus $C$ is a unique mapping.
\end{proof}
The next result shows that under which conditions a mixed type additive and cubic functional equation can be stable.
\begin{theorem}\label{th4}
Let $l\in\{-1,1\}$. Let $f: X \longrightarrow Y$ be an odd mapping  satisfying
\begin{equation}\label{3.17}
\| D_f(x,y) \|  \leq \phi(x,y)
\end{equation}
for all $x,y \in X$. Suppose that $\phi:X \times X \rightarrow [0, \infty )$ is a function such that
\begin{equation}\label{3.18}
\sum_{i=1}^{\infty} 2^{il}\phi\left(\frac{x}{2^{il}},\frac{x}{2^{il}}\right)<\infty \,\,and\,\, \sum_{i=1}^{\infty} 8^{il}\phi\left(\frac{x}{2^{il}},\frac{x}{2^{il}}\right)<\infty
\end{equation}
for all $x \in X$ and that $\lim_n 2^{ln}\phi(\frac{x}{2^{ln}},\frac{y}{2^{ln}})=0$ and $\lim_n 8^{ln}\phi(\frac{x}{2^{ln}},\frac{y}{2^{ln}})=0$ for all $x,y \in X$. Then, there exist a unique additive function $A: X \longrightarrow Y$ and a unique cubic function $C: X \longrightarrow Y$ such that
\begin{equation}\label{3.19}
\left\| f(x)-A(x)-C(x) \right\| \leq \frac{1}{12}\left[\sum\limits_{i =\frac{|l-1|}{2}}^{\infty} (2^{il}+8^{il})\phi\left(\frac{x}{2^{l(i+l)}},\frac{x}{2^{l(i+l)}}\right)\right]
\end{equation}
for all $x \in X$.
\end{theorem}
\begin{proof}
By Theorem \ref{th2} and Theorem \ref{th3}, there exist a unique additive mapping $A_0: X \longrightarrow Y$ and a unique cubic mapping $C_0: X \longrightarrow Y$ such that 
\begin{equation}\label{3.20}
\left\| f(2x)-8f(x)-A_0(x)\right\| \leq \frac{1}{2}\sum\limits_{i =\frac{|l-1|}{2}}^{\infty} 2^{il}\phi\left(\frac{x}{2^{l(i+l)}},\frac{x}{2^{l(i+l)}}\right)
\end{equation}
and
\begin{equation}\label{3.21}
\left\| f(2x)-2f(x)-C_0(x)\right\| \leq \frac{1}{2}\sum\limits_{i =\frac{|l-1|}{2}}^{\infty} 8^{il}\phi\left(\frac{x}{2^{l(i+l)}},\frac{x}{2^{l(i+l)}}\right)
\end{equation}
for all $x \in X$. Plugging (\ref{3.20}) into (\ref{3.21}) to obtain 
\begin{equation}\label{3.22}
\left\| f(x)+\frac{1}{6}A_0(x)-\frac{1}{6}C_0(x)\right\| \leq \frac{1}{12}\left[\sum\limits_{i =\frac{|l-1|}{2}}^{\infty} (2^{il}+8^{il})\phi\left(\frac{x}{2^{l(i+l)}},\frac{x}{2^{l(i+l)}}\right)\right]
\end{equation}
Putting $A(x)=-\frac{1}{6}A_0(x)$ and $C(x)=\frac{1}{6}C_0(x)$ in (\ref{3.22}), we get (\ref{3.19}). 
\end{proof}


In the following corollaries, we establish the Hyers-Ulam-Rassias stability problem for functional equation (\ref{1.4}).
\begin{corollary}
Let $p$ and $\theta$ be nonnegative integer numbers with $p\neq 1,3$. Suppose that an odd mapping $f: X \longrightarrow Y$ satisfies
\begin{align*}
\left\|D_f (x,y)\right\| \leq \theta\left(\left\|x\right\|^p + \left\|y\right\|^p\right)
\end{align*}
for all $x,y \in X$. Then there exist a unique additive function $A: X \longrightarrow Y$ and a unique cubic function $C: X \longrightarrow Y$ satisfying 
\begin{align*}
\left\|f(x)-A(x)-C(x)\right\| \leq \frac{\theta}{6}\left[\frac{1}{|2^{p}-2|} + \frac{1}{|2^{p}-8|}\right]\left\|x\right\|^p
\end{align*}
for all $x \in X$.
\begin{proof}
The results follows from Theorem \ref{th4} by taking $\phi(x,y) = \theta\left(\left\|x\right\|^p + \left\|y\right\|^p\right)$.
\end{proof}
\end{corollary}


\begin{corollary}
Let $r,s$ and $\theta$ be nonnegative integer numbers with $p=r+s\neq 1,3$. Suppose that an odd mapping $f: X \longrightarrow Y$ satisfies
\begin{align*}
\left\|D_f (x,y)\right\| \leq \theta\left\|x\right\|^r\left\|y\right\|^s
\end{align*}
for all $x,y \in X$. Then there exist a unique additive function $A: X \longrightarrow Y$ and a unique cubic function $C: X \longrightarrow Y$ satisfying 
\begin{align*}
\left\|f(x)-A(x)-C(x)\right\| \leq \frac{\theta}{12}\left[\frac{1}{|2^{p}-2|} + \frac{1}{|2^{p}-8|}\right]\left\|x\right\|^p
\end{align*}
for all $x \in X$.
\begin{proof}
Choosing $\phi(x,y) = \theta\left\|x\right\|^r\left\|y\right\|^s$ in Theorem \ref{th4}, one can obtain the desired result.
\end{proof}
\end{corollary}


\subsection*{Acknowledgements}
The authors would like to thank the anonymous reviewer for very helpful comments and suggesting some related references.

\end{document}